\documentclass[a4paper,12pt]{amsart}

\usepackage{amsfonts}
\usepackage{amsmath}
\usepackage{amssymb}
\usepackage{mathrsfs}
\usepackage{hyperref}
\usepackage{graphicx}
\usepackage{enumitem}
\usepackage[usenames]{color}

\newtheorem{thm}{Theorem}[section]
\newtheorem{cor}[thm]{Corollary}
\newtheorem{lem}[thm]{Lemma}

\numberwithin{equation}{section}

\theoremstyle{definition}
\newtheorem{definition}[thm]{Definition}
\newtheorem{rem}[thm]{Remark}
\newtheorem{exa}[thm]{Example}

\begin{document}

\title[ Existence and uniqueness of some Cauchy Type Problems  ]{Existence and uniqueness of some Cauchy Type Problems in fractional $q$-difference calculus}

\author[Lars-Erik Persson]{L.E. Persson}
\address{
  Lars-Erik Persson:
  \endgraf
  Department of Computer Science and Computational Engineering
  \endgraf
  UiT The Artic University of Norway, Campus Narvik
  \endgraf
  Narvik, Norway
  \endgraf
   and
  \endgraf
  Department of computer science and mathematics, Karlstad university
  \endgraf
  Karlstad, Sweden.
  \endgraf
  {\it E-mail address} {\rm larserik6pers@gmail.com}
  }

\author[Serikbol Shaimardan]{S. Shaimardan}
\address{
  Serikbol Shaimardan:
  \endgraf
  L. N. Gumilyev Eurasian National University
  \endgraf
  5 Munaytpasov str., Astana, 010008
  \endgraf
  Kazakhstan
  \endgraf
  {\it E-mail address} {\rm shaimardan.serik@gmail.com}
  }
\author[Nariman Sarsenovich Tokmagambetov]{N.S. Tokmagambetov}
\address{
  Nariman Sarsenovich Tokmagambetov:
  \endgraf
    L. N. Gumilyev Eurasian National University
  \endgraf
   5 Munaytpasov str., Astana, 010008
  \endgraf
  Kazakhstan,
  \endgraf
    {\it E-mail address} {\rm nariman.tokmagambetov@gmail.com}
 }

\date{}

\begin{abstract}
In this paper we derive a sufficient condition for the existence of a unique
  solution of a Cauchy type $q$-fractional problem (involving the fractional $q$-derivative of Riemann-Liouville type)
   for some nonlinear differential equations. The key technique is to first prove that this Cauchy type $q$-fractional problem is equivalent to a corresponding Volterra $q$-integral equation. Moreover, we define the $q$-analogue of the Hilfer fractional derivative or composite fractional derivative operator and prove some similar new equivalence, existence and uniqueness results as above. Finally, some examples are presented to illustrate our main results in cases where we can even give concrete formulas for these unique solutions.
\end{abstract}



\maketitle

{\section{Introduction}}

Nonlinear fractional differential equations   play important roles  due to their numerous applications and also for the important role they play  not only in mathematics but also in other sciences. In particular, they   arise naturally in real world phenomena related to physics, chemistry, biology,  signal-and image processing. Moreover, they are equipped with social sciences such as food supplement, climate and economics, see e.g. \cite{HP2006, H2000}. Therefore during the last twenty years, there has been a significant development in ordinary and partial differential equations involving fractional derivatives and a huge amount of papers and also some books devoted to this subject in various spaces have appeared, see e.g. the monographs of T. Sandev  and  Z. Tomovski \cite{TZ2019}, A.A. Kilbas et al. \cite{KST2006}, R. Hilfer \cite{H2000}, K.S. Miller and the B. Ross \cite{MR1993}, the papers  \cite{H2002}, \cite{T2012}, \cite{BG2003}, \cite{B2008}, \cite{L2008},  \cite{K2009}, \cite{LEP}, \cite{Shaimardan15}, and  \cite{Shaimardan16} and the references therein.

The origin of the  $q$-difference calculus can be traced back to the works  in \cite{J1908, J1910} by  F. Jackson and R.D. Carmichael \cite{C1912} from the beginning of the twentieth century, while basic definitions and properties can be found e.g. in the monographs \cite{CK2000, E2002}.  Recently,  the fractional $q$-difference calculus has been proposed by W. Al-salam  \cite{A1966} and R.P. Agarwal \cite{A1969}. Today, maybe due to the explosion in research within the fractional differential calculus setting, new developments in this theory of fractional $q$-difference calculus have been addressed extensively by several researchers. For example, some researchers obtained $q$-analogues of the integral and differential fractional operators properties such as the $q$-Laplace transform and $q$-Taylor’s formula \cite{PMS2007}, $q$-Mittag-Leffler function \cite{A1966}  and so on.

We also pronounce that up to now much attention have been focused on the fractional $q$-difference equations. There have been some papers dealing with the existence and uniqueness or multiplicity of solutions for nonlinear fractional q-difference equations by the use of some well-known fixed point theorems. For some recent developments on the subject, see e.g. \cite{AM2012}, \cite{ZCZ2013}, \cite{F2011}, \cite{F2010} and the references therein. In Section 3 of this paper we continue such research by proving some new results with focus on uniqueness. The main result in this Section is Theorem \ref{thm 3.2} but in order to prove this result we need to prove two results (Theorem \ref{thm 3.1} and Lemma \ref{lem 2.6}) of independent interest.

The notations used in this introduction are explained in our Section 2 below. In this paper,  we also focus more on the $q$-analogue of the Hilfer fractional derivative or composite fractional derivative operator (see Definition \ref{definition4.1} and Definition \ref{definition4.2}) and we derive a sufficient conditions for the existence of a unique solution of a Cauchy type $q$-fractional problem:
\begin{eqnarray}\label{additive1.1}
\left(D^{\alpha,\beta}_{q,a+}y\right)(x) &=& f\left(x,y(x)\right),\;\;\;n-1<\alpha\leq n; n\in\mathbb{N}, 0\leq \beta\leq 1,
\end{eqnarray}
\begin{eqnarray}\label{additive1.2}
\lim\limits_{x\rightarrow a+}\left(D^k_qI^{(n-\alpha)(1-\beta)}_{q,a+}y \right)(x)=b_k,  b_k\in\mathbb{R},  k=0,1,2,\dots n-1,
\end{eqnarray}
and as a particular case of this nonlinear model we have
\begin{eqnarray}\label{additive1.3}
\left(\mathcal{D}^{\alpha,\beta}_{q,a+}y\right)(x) &=& f\left(x,y(x)\right), \;\;\;0<\alpha\leq 1; 0\leq \beta\leq 1,
\end{eqnarray}
\begin{eqnarray}\label{additive1.14}
\lim\limits_{x\rightarrow a+}\left(\mathcal{D}^k_qI^{(1-\alpha)(1-\beta)}_{q,a+}y \right)(x)=b_k , b_k\in\mathbb{R}, k=0,1,2,\dots n-1,
\end{eqnarray}
where $\mathcal{D}^{\alpha,\beta}_{q,a+}$ and  $D^{\alpha,\beta}_{q,a+}$  are the  Hilfer fractional $q$-derivative or composite fractional $q$-derivative operators  and $f(.,.): [a,b]\times\mathbb{R}\rightarrow \mathbb{R}$, $0<a<b<\infty$ (see Theorem \ref{thm 4.2}). Moreover, for the proof of this theorem we prove an equivalence theorem (Theorem \ref{thm 4.1}) of independent interest. In the case when $\beta = 0$ this is the generalized Riemann–Liouville   fractional  $q$-derivative ( see Definition \ref{definition 2.2}) and in case when $\beta = 1$ it corresponds to the Caputo fractional $q$-derivative (see \cite{PMS2009}). We also prove a Lemma (Lemma 4.4) of independent interest.

The paper is organized as follows: The main results are presented and proved in Section 3  and Section 4 and the announced examples are given in Section 5. In order to not disturb these presentations we include in Section 2 some necessary Preliminaries. In particular, we state and prove a necessary lemma (Lemma \ref{lem 2.6}) of independent interest.

\;\;\;\;\;\;\;\;\;\;


{\section{Preliminaries}}

First we recall some elements of $q$-calculus, for more information see e.g. the books \cite{CK2000},  \cite{E2002} and \cite{AM2012}. Throughout this paper, we assume that $0<q<1$ and  $ 0\leq a < b <\infty $.

Let $\alpha \in \mathbb{R}$. Then a $q$-real number $[\alpha ]_q$ is defined by
$$
[\alpha ]_{q}:=\frac{1-q^{\alpha }}{1-q},
$$
where $\mathop{\lim }\limits_{q\rightarrow 1}\frac{1-q^{\alpha }}{1-q}%
=\alpha $.

We introduce for $k\in\mathbb{N}$:
\begin{eqnarray*}
(a;q)_0=1,  \; (a;q)_n=\prod\limits_{k=0}^n\left(1-q^ka\right),  \; (q;a)_\infty = \lim\limits_{n\rightarrow\infty}(a,q)^n_q, \; and \;
(a;q)_\alpha=\frac{(a;q)_\infty}{(q^\alpha a;q)_\infty}.
\end{eqnarray*}

The $q$-analogue of the power function $(a- b)_q^\alpha$ is defined by
 \begin{eqnarray*}
(a-b)_q^\alpha{:=}a^\alpha\frac{(\frac{a}{b};q)_\infty}{(q^\alpha\frac{a}{b};q)_\infty} .
\end{eqnarray*}

Notice that $(a-b)_q^\alpha=a^\alpha(\frac{a}{b};q)_\alpha$ .

The $q$-analogue of the binomial coefficients $[n]_{q}!$ are defined by
\begin{equation*}
[n]_{q}!:=\left\{
\begin{array}{l}
{1,\mathrm{\;\;\;\;\;\;\;\;\;\;\;\;\;\;\;\;\;\;\;\;\;\;\;\;\;\;\;\;\;\;\;\;%
\;if\;n}=\mathrm{0,}} \\
{[1]_{q}\times [2]_{q}\times \cdots \times [n]_{q},\mathrm{%
\;if\;n}\in \mathrm{N,\;\;}}%
\end{array}%
\right.
\end{equation*}

The gamma function $\Gamma_q(x)$ is defined by
\begin{eqnarray*}
\Gamma_q(x):=\frac{(q; q)_\infty }{(q^x; q)_\infty }(1-q)^{1-x}, \;\;\;
\end{eqnarray*}
for any $x > 0$. Moreover, it yields that
\begin{eqnarray}\label{additive2.0}
\Gamma_q(x)[x]_q=\Gamma_q(x+1).
\end{eqnarray}

The $q$-analogue differential operator $D_{q}f(x)$ is
\begin{eqnarray*}
D_{q}f(x):=\frac{f(x)-f(qx)}{x(1-q)},
\end{eqnarray*}
and the $q$-derivatives $D^{n}_q(f(x))$ of higher order are defined inductively as follows:
\begin{eqnarray*}
D_q^0(f(x)):=f(x), \;\;\; D^{n}_q(f(x)):=D_q\left(D_q^{n-1}f(x)\right), (n=1,2,3, \dots)
\end{eqnarray*}

Notice that
\begin{eqnarray}\label{additive2.1}
D_{q}\left[(x-b)_q^\alpha\right]&=&[\alpha]_q(x-b)_q^{\alpha-1}
\end{eqnarray}
and
\begin{eqnarray}\label{additive2.2}
D_{q}\left[(a-x)_q^\alpha\right]&=&-[\alpha]_q(a-qx)_q^{\alpha-1}.
\end{eqnarray}

The $q$-integral (or Jackson integral) $\int\limits_a^b f(x)d_{q}x$ is defined by
\begin{eqnarray}\label{additive2.3}
\int\limits_0^a f(x)d_{q}x:=(1-q)a\sum\limits_{m=0}^\infty q^{m}f(aq^{m})
\end{eqnarray}
for $a=0$ and
\begin{eqnarray}\label{additive2.4}
\int\limits_a^b f(x)d_{q}x=\int\limits_0^b f(x)d_{q}x-
\int\limits_0^a f(x)d_{q}x,
\end{eqnarray}
for $0<a<b$. Notice that
\begin{eqnarray}\label{additive2.5}
\int\limits_a^b D_{q}f(x)d_{q}x=f(b)-f(a)
\end{eqnarray}
and
\begin{eqnarray}\label{additive2.6}
\int\limits_a^b\int\limits_a^x f(x)g(t)d_{q}td_{q}x&=&\int\limits_a^b\int\limits_{qt}^b f(x)g(t)d_{q}xd_{q}t.
\end{eqnarray}

Moreover, the multiple $q$-integral $\left(I^n_{q,a+}f\right)(x)$ is
\begin{eqnarray}\label{additive2.7}
\left(I^n_{q,a+}f\right)(x)&:=&
\int\limits_a^x\int\limits_a^t\int\limits_a^{t_{n-1}}\dots\int\limits_a^{t_2}d_qt_1d_qt_2{\dots}d_qt_{n-1}d_qt\nonumber\\
&=&\frac{1}{\Gamma_q(n)}\int\limits_a^x(x-qt)_q^{n-1}f(t)d_qt.
\end{eqnarray}

\begin{definition}\label{definition2.1}
The Riemann-Liouville $q$-fractional integrals $I^\alpha_{a+}f$ of order $\alpha>0$
are defined by
\begin{eqnarray}\label{additive2.8A}
\left(I^\alpha_{q,a+}f\right)(x):=\frac{1}{\Gamma_q(\alpha)}\int\limits_a^x(x-qt)_q^{\alpha-1}f(t)d_qt.
\end{eqnarray}
\end{definition}

\begin{definition}\label{definition 2.2}The Riemann-Liouville fractional $q$-derivative $D^\alpha_{q,a+}f$ of order $\alpha>0$ is defined by
\begin{eqnarray*}
\left(D^\alpha_{q,a+}f\right)(x):=\left( D^{[\alpha]}_{q,a+}I^{[\alpha]-\alpha}_{q,a+}f\right)(x).
\end{eqnarray*}
\end{definition}

Notice that
\begin{eqnarray}\label{additive2.8}
\left(I^\alpha_{q,a+}(t-a)_q^\lambda\right)(x)=
\frac{\Gamma_q(\lambda+1)}{\Gamma_q(\alpha+\lambda+1)} (x-a)_q^{\alpha+\lambda},
\end{eqnarray}
for $ \lambda\in(-1,\infty)$.

For $1\leq p<\infty$ we  define the space $L^p_q=L^p_q [a, b]$ by
\begin{eqnarray*}
L^p_q [a, b] :=\left\{f:[a,b]\rightarrow \mathbb{ R}:\left(\int\limits_a^b |f(x)|^pd_qx\right)^\frac{1}{p}<\infty\right\}.
\end{eqnarray*}

\begin{lem}\label{lem 2.3} a) Let $\alpha>0$, $\beta>0$ and $1\leq p<\infty$. Then the $q$–fractional integrals has the following
semigroup property
$$
\left(I^\alpha_{q,a+}I^\beta_{q,a+}f\right)(x) =\left(I^{\alpha+\beta}_{q,a+}f\right)(x),
$$
for all  $x \in[a, b]$ and $f(x) \in L^p_q[a,b]$.

b) Let $\alpha>\beta > 0$, $1 \leq p<\infty$ and $f(x) \in L^p_q[a,b]$. Then the following
equalities
$$
\left(D^\alpha_{q,a+}I^\alpha_{q,a+}\right)(x)=f(x),\;\;\;\left(D^\beta_{q,a+}I^{\alpha }_{q,a+}f\right)(x)=\left(I^{\alpha-\beta}_{q,a+}f\right)(x),
$$
hold for all  $x\in[a,b]$.
\end{lem}

\begin{proof}
 a) The proof for the case $p=1$ can be found in
 \cite[Theorem 5]{PMS2007}. The proof for the case $p>1$ is completely similar so we leave out the details.

 b)This statement was proved in \cite[Lemma 9]{PMS2007}.
\end{proof}

\begin{definition}\label{definition 2.5}  A function  $f:[a,b]\rightarrow \mathbb{ R}$ is called $q$-absolutely continuous
if $\exists \varphi\in L^1_q[a,b]$ such that
\begin{eqnarray}\label{additive2.9}
f(x)=f(a)+\int\limits_a^x\varphi(t)d_qt,
\end{eqnarray}
for all $x\in[a,b]$.
\end{definition}

The collection of all $q$-absolutely continuous functions on $[a,b]$ is denoted  $AC_q[a,b]$. For $n \in N := {1,2,3,\dots}$ we denote by $AC^n_q[a, b]$ the space of real-valued functions $f(x)$ which have $q$-derivatives up to order $n-1$ on $[a, b]$ such
that $D_q^{n-1}f(x)\in AC_q[a,b]$:
$$
AC^n_q[a, b]:=\left\{f:[a,b]\rightarrow \mathbb{ R}; D_q^{n-1}f(x)\in AC_q[a,b] \right\}.
$$

\begin{lem}\label{lem 2.5}  a) A function $f:[a,b]\rightarrow \mathbb{ R}$ belongs to $AC^n_q[a, b]$ if the following equality holds
\begin{eqnarray}\label{additive2.10}
f(x) &=& \frac{1}{[n-1]_q!}\int\limits_a^x(x-qt)_a^{n-1}\varphi(t)d_t+\sum\limits_{k=0}^{n-1}c_k(x-a)_q^k,
\end{eqnarray}
where $\varphi(x):=D^{n}_qf(x)$ and $c_k=\frac{D^{k}_qf(a)}{\Gamma_q(k)}, k=0,1,2,\dots, n-1$, are constants.

b) Let $f(x)\in L_q^1[a,b]$ and $ \left(I^{n-\alpha}_{q,a+}f\right)(x) \in AC^n_q[a,b]$ with $n=[\alpha], \alpha>0$. Then the following equality holds:

$$
\left(I^\alpha_{q,a+}D_{q,a+}^\alpha f\right)(x)=f(x)-\sum\limits_{k=0}^{n-1}\frac{\left(D^{\alpha-k}_{q,a+}f\right)(a)}{\Gamma_q\left(\alpha-k+1\right)}(x-a)_q^{\alpha-k},
$$
while for $[a] = n \in \mathbb{N}$, it yields that
\begin{eqnarray*}
\left(I^n_{q,a+}D_{q,a+}^nf\right)(x)&=&f(x)-\sum\limits_{k=0}^{n-1}\frac{D^{\alpha-k}_qf(a)}{[k]_q!}(x-a)_q^k\\
&=&f(x)-\sum\limits_{k=0}^{n-1}\frac{b_k}{[k]_q!}(x-a)_q^k.
\end{eqnarray*}
\end{lem}

\begin{proof} a) The proof follows directly from the definition of $AC^n_q[a, b]$, (\ref{additive2.9}) and (\ref{additive2.5}) (c.f. also \cite[Theorem 20.2]{CK2000}) so we leave out the details.

b) This statement was proved in \cite[ Lemma 4.17]{AM2012} when $a=0$, but the proof is the same for $a\neq0$.
\end{proof}

We also need the following result of independent interest:
\begin{lem}\label{lem 2.6}  Let $\alpha>0$ and $1\leq p<\infty$.  Then the fractional integration operator $I^\alpha_{q,a+} $ is bounded in
$L^p_q [a, b]$:
\begin{eqnarray}\label{additive2.7}
\|I^\alpha_{q,a+}f\|_{L^p_q [a, b]} &\leq& K\|f\|_{L^p_q [a, b]},
\end{eqnarray}
where $K:=\frac{(b-qa)_q^\alpha}{ \Gamma_q(\alpha+1)}$.
\end{lem}

\begin{proof} For $p>1$,  $p'$ is as usual defined by $\frac{1}{p}+\frac{1}{p'}=1$.  Then, using the H\"{o}lder-Rogers inequality,
Definition 2.1  and (\ref{additive2.1}),  (\ref{additive2.2}), (\ref{additive2.5}) and (\ref{additive2.6}),  we obtain that
\begin{eqnarray*}
\|I^\alpha_{q,a+}f\|_{L^p_q [a, b]}
&\leq& \frac{\left(\int\limits_a^b
\left[\int\limits_a^x(x-qt)_q^{\alpha-1} \left|f(t)\right| d_qt\right]^pd_qx\right)^\frac{1}{p}}{\Gamma_q(\alpha)} \\
&\leq&\frac{ \left(\int\limits_a^b \left[\int\limits_a^x(x-qt)_q^{\alpha-1}d_qt\right]^\frac{p}{p'}
\left[\int\limits_a^x(x-qt)_q^{\alpha-1}\left|f(t)\right|^pd_qt\right] d_qx\right)^\frac{1}{p}}{\Gamma_q(\alpha)}
 \end{eqnarray*}
 \begin{eqnarray*}
&\leq&\frac{ \left(\int\limits_a^b \left[\frac{\int\limits_a^xD_{q,t}\left[(x-qt)_q^\alpha\right] d_qt}{-[\alpha]_q}\right]^\frac{p}{p'}
\left[\int\limits_a^x(x-qt)_q^{\alpha-1}\left|f(t)\right|^pd_qt\right] d_qx\right)^\frac{1}{p}}{\Gamma_q(\alpha)} \\
&\leq&\frac{ \left[\frac{(b-a)_q^\alpha}{[\alpha]_q}\right]^\frac{1}{p'} \left(\int\limits_a^b\left|f(t)\right|^p
 \int\limits_{qt}^b(x-qt)_q^{\alpha-1}d_qxd_qt\right)^\frac{1}{p}}{\Gamma_q(\alpha)} \\
 \end{eqnarray*}
 \begin{eqnarray*}
 &\leq&\frac{ \left[\frac{(b-a)_q^\alpha}{[\alpha]_q}\right]^\frac{1}{p'}}{\Gamma_q(\alpha)}
 \left(\frac{\int\limits_a^b\left|f(t)\right|^p
 \int\limits_{qt}^bD_{q,x}\left[(x-qt)_q^\alpha\right]d_qxd_qt}{[\alpha]_q}\right)^\frac{1}{p}\\
 &\leq&\frac{\left[\frac{(b-qa)_q^\alpha}{[\alpha]_q}\right]^\frac{1}{p'}\left[\frac{(b-qa)_q^\alpha}{[\alpha]_q}\right]^\frac{1}{p}}{\Gamma_q(\alpha)}
  \|f(t)\|_{L^p_q [a, b]}\\
 &\leq& K \|f(t)\|_{L^p_q [a, b]},
\end{eqnarray*}
where $(b-a)_q^\alpha\leq (b-qa)_q^\alpha$  for $\alpha>0$. The proof is complete.
\end{proof}


{\section{On the solutions of some fractional $q$-differential equations with the Riemann-Liouville fractional
$q$-derivative.}}

In this  section we will consider the   nonlinear model with Riemann-Liouville fractional
$q$-derivative:

\begin{eqnarray}\label{additive3.1A}
\left(D^\alpha_{q,a+}y\right)(x) &=& f\left(x,y(x)\right),\;\;\;n-1<\alpha\leq n; n\in\mathbb{N},
\end{eqnarray}
\begin{eqnarray}\label{additive3.2A}
\left(D^{\alpha-k}_{q,a+}y \right)(a+)=b_k, b_k\in\mathbb{R}, k=0,1,2, \dots,  n-1.
\end{eqnarray}

In the classical case the investigations in this direction involve the existence and uniqueness of solutions to fractional differential equations with the Riemann-Liouville fractional derivative. Several authors have considered such problems even in nonlinear cases, see e.g. \cite[Section 3]{KST2006} and the references therein. Here we use another approach and first prove an equivalence theorem of independent interest.
\vspace{0,3cm}

\subsection{ Equivalence of the Cauchy type $q$-fractional problem and a $q$-Volterra
integral equation.}

\begin{thm}\label{thm 3.1}  Let $n-1<\alpha\leq n; n\in\mathbb{N}$, $G$ be an open set in $\mathbb{R}$ and
$f(.,.):  ( a , b] \times G\rightarrow \mathbb{R}$ be a function such that $f(x,y(x)) \in L_q^1(a, b)$ for any $y \in G$.
If $y(x)\in L_q^1(a, b)$, then $y(t)$ satisfies a.e. the relations (\ref{additive3.1A})-(\ref{additive3.2A}) if and only if $y(x)$ satisfies a.e. the integral equation
\begin{eqnarray}\label{additive3.1}
y(x):=\sum\limits_{k=0}^{n-1}\frac{b_k}{\Gamma_q\left(\alpha-k+1\right)}(x-a)_q^{\alpha-k}+
\left[I^\alpha_{q,a+}f(t,y(t))\right](x).
\end{eqnarray}
\end{thm}

\begin{proof} \emph{Necessity}. Let $n-1<\alpha\leq n; n\in\mathbb{N}$ and $y(x)\in L_q^1(a, b)$ satisfy a.e. the relations (\ref{additive3.1A})-(\ref{additive3.2A}). Since
$f(t,y) \in L_q^1(a, b)$ and  (\ref{additive3.1}) we find that  $\exists \left(D^\alpha_{q,a+}y\right)(x)\in L_q^1(a, b)$.
Then, by using Definition \ref{definition 2.2} we have that $\left(D^\alpha_{q,a+}y\right)(x) =D_q^n\left[I^{n-\alpha}_{q,a+}y\right](x)$ and
$$
\int\limits_a^xD_q^n\left[I^{n-\alpha}_{q,a+}y\right](x)d_qx=D_q^{n-1}\left[I^{n-\alpha}_{q,a+}y\right](x)-
D_q^{n-1}\left[I^{n-\alpha}_{q,a+}y\right](a).
$$

Hence, according to Lemma \ref{lem 2.5} (a), $\left(I^{n-\alpha}_{q,a+}y\right)(x)\in AC^n_q[a,b]$. Thus, we can apply  Lemma \ref{lem 2.5} (b) and (\ref{additive3.2A}) to conclude that
\begin{eqnarray}\label{additive3.2}
\left(I^\alpha_{q,a+}D_q^\alpha y\right)(x)&=&y(x)-\sum\limits_{k=0}^{n-1}\frac{\left(D^{\alpha-k}_{q,a+}y\right)(a)}{\Gamma_q\left(\alpha-k+1\right)}(x-a)_q^{\alpha-k}\nonumber\\
&=&y(x)-\sum\limits_{k=0}^{n-1}\frac{b_k}{\Gamma_q\left(\alpha-k+1\right)}(x-a)_q^{\alpha-k}.
\end{eqnarray}

Moreover, by Lemma \ref{lem 2.6}, the integral $\left[I^\alpha_{q,a+}f(.,.)\right](x)\in L_q^1[a,b]$. Finally, by
applying the operator $I^\alpha_{q,a+}$ to both sides of (\ref{additive3.1A}) and using (\ref{additive3.2}) and (\ref{additive2.5}),
we obtain the equation (\ref{additive3.1}), and hence the necessity is proved.

\emph{Sufficiency.}  Let $y(x) \in L_q^1(a, b)$ and assume that it satisfies the equation (\ref{additive3.1}). Then, by applying the operator $D_{q,a+}^\alpha$ to both sides of (\ref{additive3.1}), we have that
\begin{eqnarray}\label{additive3.3}
\left(D_{q,a+}^\alpha y\right)(x)&=&\sum\limits_{k=0}^{n-1}\frac{b_k}{\Gamma_q\left(\alpha-k+1\right)}
\left(D_{q,a+}^\alpha(t-a)_q^{\alpha-k}\right)(x)\nonumber\\
&+&\left[D_{q,a+}^\alpha I^\alpha_{q,a+} f(t,y(t))\right](x).
\end{eqnarray}

From (\ref{additive2.8}) it follows that $\left(I^{n-\alpha}_{q,a+}(t-a)_q^{\alpha-k}\right)(x)=
\frac{\Gamma_q\left(\alpha-k+1\right)}{\Gamma_q\left(n-k+1\right)}(t-a)_q^{n-k}$. Furthermore, it yields that
\begin{eqnarray}\label{additive3.4}
\frac{\left(D_{q,a+}^\alpha(t-a)_q^{\alpha-k}\right)(x)}{\Gamma_q\left(\alpha-k+1\right)} &=& \frac{\left(D_q^n I^{n-\alpha}_{q,a+}(t-a)_q^{\alpha-k}\right)(x)}{\Gamma_q\left(\alpha-k+1\right)}\nonumber\\
&=&\frac{\left(D_q^n  (t-a)_q^{n-k}\right)(x)}{\Gamma_q\left(n-k+1\right)}=0.
\end{eqnarray}

Consequently, combining (\ref{additive3.3}) and (\ref{additive3.4}) and by using Lemma \ref{lem 2.3} (b)
we arrive at the equation (\ref{additive3.1A}).

Now we show that the relations in (\ref{additive3.2A}) also hold. By Definition \ref{definition 2.2} and (\ref{additive2.1}) and (\ref{additive2.8}) we get that
\begin{eqnarray}\label{additive3.5}
\left[D_{q,a+}^{\alpha-m }(t-a)_q^{\alpha-k}\right](x)&=&\left[D_q^{n-m }I^{n -\alpha}_{q,a+}(t-a)_q^{\alpha-k}\right](x)\nonumber\\
&=&\frac{\Gamma_q\left(\alpha-k+1\right)}{\Gamma_q\left(n -k+1\right)}\left[D_q^{n-m }(t-a)_q^{n -k}\right](x)\nonumber\\
&=&\frac{\Gamma_q\left(\alpha-k+1\right)}{\Gamma_q\left(n -k+1\right)}\left[D_q^{n-m }(t-a)_q^{n -k}\right](x)\nonumber\\
&=&\frac{\Gamma_q\left(\alpha-k+1\right)}{\Gamma_q\left(m-k+1\right)} (x-a)_q^{m -k}.
\end{eqnarray}

Next we apply the operators $D_{q,a+}^{\alpha-m}$ with $m = 1,\dots , n-1)$ to both sides of (\ref{additive3.1})  and Lemma \ref{lem 2.3} (b) and
(\ref{additive3.5}) to conclude that
\begin{eqnarray*}
\left(D_{q,a+}^{\alpha-m }y\right)(x)&=&\sum\limits_{k=0}^{n-1}\frac{b_k\left[D_{q,a+}^{\alpha-m }(t-a)_q^{\alpha-k}\right](x)}
{\Gamma_q\left(\alpha-m+1\right)}
+ \left[D_{q,a+}^{\alpha-m }I^\alpha_{q,a+}f(t,y(t))\right](x)\nonumber\\
&=&\sum\limits_{k=0}^{n-1}\frac{b_k}{\Gamma_q\left(m-k+1\right)} (x-a)_q^{m-k} +
 \left[I^m_{q,a+}f(t,y(t))\right](x).
\end{eqnarray*}

Finally, by taking the limit of (\ref{additive3.1}) when $x\rightarrow a+$,  we obtain the relations in
(\ref{additive3.2A}). The proof is complete.

\end{proof}

\subsection{  Existence and uniqueness of a unique global solution to the Cauchy type $q$-fractional problem (\ref{additive3.1A})-(\ref{additive3.2A}) in $L^1_{\alpha, q}[a,b]$}.

\vspace{0,3cm}
In this subsection we give conditions for a unique global solution to the Cauchy type
problem (\ref{additive3.1A})-(\ref{additive3.2A}) in the space $L^1_{\alpha, q}[a,b]$ defined for $\alpha>0$ by
\begin{eqnarray*}
L^1_{\alpha, q}[a,b]:=\left\{y\in L_q^1[a, b]: D_{q,a+}^\alpha\in L_q^1[a, b]\right\}.
\end{eqnarray*}

The proof of the following existence and uniqueness theorem depends heavily on Theorem \ref{thm 3.1} and the
Banach fixed point Theorem (see e.g. \cite{KF1968}).

\begin{thm}\label{thm 3.2}  Let $n-1<\alpha\leq n; n\in\mathbb{N}$ and  $G\subset \mathbb{R}$ be an open set  and  $f(.,.) :
[a,b]\times G \rightarrow \mathbb{R}$ be a function such that $f(x,y(x))\in L_q^1[a,b]$ for any $y\in G$ and
satisfies a Lipschitz condition  in the following form:
\begin{eqnarray}\label{additive3.6}
\left|f(x,y_1(x))-f(x,y_2(x))\right|\leq C\left|y_1(x)-y_2(x)\right|,
\end{eqnarray}
where $C>0$. Then there exists a unique solution $y(x)\in L_q^1[a,b]$ to the Cauchy type problem (\ref{additive3.1A})-(\ref{additive3.2A}).
\end{thm}

\begin{proof} According to Theorem \ref{thm 3.1}, it is sufficient to study the existence of a unique solution
$y(x) \in L[a,b]$ to the $q$-integral equation (\ref{additive3.1}).
Consequently, the equation (\ref{additive3.1}) can be written in the operator form $y=Fy$ such that
\begin{eqnarray}\label{additive3.7}
\left(Fy\right)(x):=y_0+\left[I^\alpha_{q,a+}f(t,y(t))\right](x),
\end{eqnarray}
where $y_0:=\sum\limits_{k=0}^{n-1}\frac{b_k}{\Gamma_q\left(\alpha-k+1\right)}(x-a)_q^{\alpha-k}$.

Let $[a,\xi] \subset [a,b]$  be such that it holds that
\begin{eqnarray}\label{additive3.8}
\omega:=C\frac{\left(a-q\xi\right)_q^\alpha}{\Gamma_q\left(\alpha+1\right)}\leq1.
\end{eqnarray}

From (\ref{additive3.7}), (\ref{additive3.8}) and Lemma \ref{lem 2.6}  it follows that
\begin{eqnarray*}
\| Fy_1-Fy_2\|_{L_q^1[a,\xi]}\leq C\|I^\alpha_{q,a+}\left(y_1-y_2\right)\|_{L_q^1[a,\xi]}\leq \omega\| y_1-y_2 \|_{L_q^1[a,\xi]}.
\end{eqnarray*}

Hence, according to the Banach fixed point theorem (see e.g.\cite{KF1968}), there exists a unique solution $y'\in L_q^1[a,\xi]$ such that $Ty'=y'$.

Moreover, in view of this theorem, the solution $y'$ is obtained as a limit of a convergent sequence  $\left\{\left(F^my'_0\right)(x)\right\}$ in the space $L_q^1[a,\xi]$, i.e. that
\begin{eqnarray*}
\lim\limits_{m\rightarrow0}\|F^my'_0-y'\|_{L_q^1[a,\xi]}=0,
\end{eqnarray*}
where $ y_0$ is any function in $L_q^1[a,\xi]$.

If at least one $b_k\neq0$ in the initial condition
(\ref{additive3.2A}), then we can take $y'_0=y_0$.  By (\ref{additive3.8}), the sequence $\left\{\left(F^my'_0\right)(x)\right\}$ is defined by
$$
\left(F^my'_0\right)(x):=y_0+\left[I^\alpha_{q,a+}f(t,F^{m-1}y'_0(t))\right](x),
$$
for $m\in\mathbb{N}$. Let $\left(y_m(x):=F^my'_0\right)(x)$. Then
$y_m(x)=y_0+\left[I^\alpha_{q,a+}f(t,y_{m-1} (t))\right](x)$ and
$$
\lim\limits_{m\rightarrow0}\|y_m-y'\|_{L_q^1[a,\xi]}=0.
$$
This means that we actually applied the method of successive approximations to
find a unique solution $y'(x)$ to the integral equation (\ref{additive3.1}). Thus,
there exists a unique solution $y(x) = y'(x) \in L_q^1[a,b]$ to the   equation (\ref{additive3.3}) and hence to the Cauchy type problem
(\ref{additive3.1A})-(\ref{additive3.2A}).
 This fact completes the proof of Theorem 3.2.

\end{proof}

\section{ On the solutions of some fractional $q$-differential equations with the
composite fractional $q$-derivative}

In this section we define the  $q$-analogue of the composite fractional operator or Hilfer  derivative operator (see \cite{H2000}, \cite{H2002}).
Moreover, the existence and uniqueness theorems for  nonlinear fractional $q$-differential equations with Hilfer fractional
$q$-derivatives types will be proved, which are the $q$-extensions of the main results given in \cite[Proposition 1, Proposition 2 and Theorem 1]{T2012} (see also \cite[Proposition 3.1, Proposition 3.2 and Theorem 3.1]{TZ2019}).

\begin{definition}\label{definition4.1} We define the Hilfer fractional  $q$-derivative   $\mathcal{D}_{q,a+}^{\alpha,\beta} f$
  of order $0 < \alpha< 1$ and type $0 \leq \beta \leq 1$ with respect to $x$ by
\begin{eqnarray}\label{additive4.1}
\left(\mathcal{D}^{\alpha,\beta}_{q,a+}f\right)(x) &:=& \left(I_{q,a+}^{\beta(1-\alpha)}D_q\left(I_{q,a+}^{(1-\beta)(1-\alpha)}f\right)\right)(x).
\end{eqnarray}
\end{definition}

Note that in case when $\beta=0$ the  generalized  fractional $q$-derivative   (\ref{additive4.1}) would correspond to the the Riemann-Liouville fractional $q$-derivative in Definition \ref{definition 2.2} and in case when $\beta= 1$ it corresponds to the Caputo fractional $q$-derivative $\left(_cD^\alpha _{q,a+}f\right)(x)$ defined by (see \cite{PMS2009}):
$$
\left(_cD^\alpha _{q,a+}f\right)(x) := \left(I_{q,a+}^{1-\alpha}D_qf\right)(x)=\frac{1}{\Gamma_q(1-\alpha)}\int\limits_a^x
\left(x-qt\right)_q^{-\alpha}f(t)d_qt.
$$

\begin{definition}\label{definition4.2}  Let $n-1<\alpha \leq n, n\in \mathbb{N}$ and $0\leq \beta\leq 1$. We define the generalized  fractional $q$-derivative   $D_{q,a+}^{\alpha,\beta} f$  as follows:
\begin{eqnarray}\label{additive4.2}
\left(D^{\alpha,\beta}_{q,a+}f\right)(x) &:=& \left(I_{q,a+}^{\beta(n-\alpha)}D^n_q\left(I_{q,a+}^{(1-\beta)(n-\alpha)}f\right)\right)(x)\nonumber\\
&=&\left(I_{q,a+}^{\beta(n-\alpha)}D_{q,a+}^{\alpha+\beta n-\alpha\beta}f\right)(x).
\end{eqnarray}
\end{definition}

For the proof of our main results in this section we also need two lemmas of independent interest.
\begin{lem}\label{lem4.3} Let $\alpha > 0$, $n=[\alpha]$ and $f\in AC^n_q[a,b]$. Then
\begin{eqnarray}\label{additive4.3B}
D^\alpha_{q,a+}f &=&\sum\limits_{k=0}^{n-1}\frac{\lim\limits_{x\rightarrow a}D_q^kf(x)}
{\Gamma_q\left(k-\alpha+1\right)}\left(x-a\right)_q^{k}+
\int\limits_a^k\frac{\left(x-qt\right)_q^{n-\alpha-1}D_q^kf(t)}{\Gamma_q(n-\alpha)}d_qt
\end{eqnarray}
for all $x\in[a,b]$.
\end{lem}
\begin{proof} Since $f\in AC^n_q[a,b]$  we have that
\begin{eqnarray}\label{additive4.3}
f(x) &=&\sum\limits_{k=0}^{n-1}\frac{\lim\limits_{x\rightarrow a}D_q^kf(x)}
{\Gamma_q\left(k+1\right)}\left(x-a\right)_q^{k}+ I^n_{q,a+}D_q^nf(x),
\end{eqnarray}
for all $x\in[a,b]$. By using Definition \ref{definition 2.2}, (\ref{additive2.1}) and (\ref{additive2.8}) we get that
\begin{eqnarray}\label{additive4.4}
D^\alpha_{q,a+}\left[\left(x-a\right)_q^k\right]&=& \left(D^\alpha_qI^{n-\alpha}_{q,a+}\left(t-a\right)_q^k\right)(x)\nonumber\\
&=&\frac{\Gamma_q(k+1)}{\Gamma_q(k+n-\alpha+1)} \left(D^n_q\left(t-a\right)_q^{k+n-\alpha}\right)(x)\nonumber\\
&=& \frac{\Gamma_q(k+1)}{\Gamma_q(k -\alpha+1)}  \left(t-a\right)_q^{k -\alpha}.
\end{eqnarray}

Applying the operator $D^\alpha_{q,a+}$ to both sides of (\ref{additive4.3}) and  using  (\ref{additive4.4}) and Lemma \ref{lem 2.3} b)  we have that
\begin{eqnarray*}
\left(D^\alpha_{q,a+}f\right)(x)&=& \sum\limits_{k=0}^{n-1}\frac{\lim\limits_{x\rightarrow a}D_q^kf(x)}
{\Gamma_q\left(k +1\right)}\left(D^\alpha_{q,a+}\left(t-a\right)_q^k\right)(x)+D^\alpha_{q,a+}\left(I^n_{q,a+}D_q^nf\right)(x)\\
&=&\sum\limits_{k=0}^{n-1}\frac{\lim\limits_{x\rightarrow a}D_q^kf(x)}
{\Gamma_q\left(k-\alpha+1\right)} \left(x-a\right)_q^{k-\alpha} + \left(I^{n-\alpha}_{q,a+}D_q^nf\right)(x),
\end{eqnarray*}
\end{proof}
so (\ref{additive4.3B}) follows from (\ref{additive2.8A}).

\begin{lem}\label{lem 4.4}  Let $y \in L_q^1(a, b)$, $n-1 <  \alpha \leq n, n \in \mathbb{N}$, $0 \leq \beta \leq 1$, $\gamma=(n-\alpha)(1-\beta)$ and
$I^\gamma_{q,a+}y \in AC_q^n[a, b] $. Then the following equality holds:
\begin{eqnarray}\label{additive4.5}
\left(I^ \alpha_{q,a+}D^{\alpha,\beta}_{q,a+}y \right)(x)&=&y(x)-y_{q,\alpha,\beta}(x),
\end{eqnarray}
where
$$
y_{q,\alpha,\beta}(x):=\sum\limits_{k=0}^{n-1}\frac{(x-a)_q^{k-\gamma}}
{\Gamma_q\left(k- \gamma+1\right)}\lim\limits_{x\rightarrow a+}\left(D_q^k I^{\gamma}_{q,a+}y\right)(x).
$$
\end{lem}
\begin{proof} From Lemma \ref{lem 2.3} a)  and Definition \ref{definition 2.2} it follows that
\begin{eqnarray}\label{additive4.6}
\left(I^ \alpha_{q,a+}D^{\alpha,\beta}_{q,a+}y \right)(x)&=&\left(I^ \alpha_{q,a+}I^{\beta(n-\alpha)}_{q,a+}
D^{\alpha+\beta n-\alpha\beta}_{q,a+}y \right)(x)\nonumber\\
&=&\left( I^{\alpha+\beta(n-\alpha)}_{q,a+}
D^{\alpha+\beta(n-\alpha)}_{q,a+}y \right)(x)\nonumber\\
&=&\left( I^{n-\gamma}_{q,a+}
D_q^nI^{\gamma}_{q,a+}y \right)(x).
\end{eqnarray}

 We define $\widetilde{y}:=I^{\gamma}_{q,a+}y \in AC_q^n[a, b]$. Then, according to  Lemma \ref{lem4.3}  and Lemma \ref{lem 2.3} b) (\ref{additive4.4}) and (\ref{additive4.6}), we find that
\begin{eqnarray*}
y(x)&=&\left(D_{q,a+}^\gamma \widetilde{y} \right)(x) \\
&=&\sum\limits_{k=0}^{n-1}\frac{\lim\limits_{x\rightarrow a}\left(D_q^k \widetilde{y}\right)(x)} {\Gamma_q\left(k-\gamma+1\right)}\left(x-a\right)_q^{k-\gamma} +\left(I^{n-\gamma}_{q,a+}D_q^n\widetilde{y}\right)(x)\\
&=&\sum\limits_{k=0}^{n-1}\frac{\lim\limits_{x\rightarrow a}\left(D_q^k I^{\gamma}_{q,a+}\right)(x)} {\Gamma_q\left(k-\gamma+1\right)}\left(x-a\right)_q^{k-\gamma} +\left(I^{n-\gamma}_{q,a+}D_q^nI^{\gamma}_{q,a+}y\right)(x)\\
&=&y_{q,\alpha,\beta}(x)+\left(I^ \alpha_{q,a+}D^{\alpha,\beta}_{q,a+}y \right)(x),
\end{eqnarray*}
which completes the proof.
\end{proof}

\subsection{ Equivalence of the Cauchy type $q$-fractional problem and a Volterra $q$-integral equation}

\begin{thm}\label{thm 4.1}  Let $G$ be an open set in $\mathbb{R}$,
$f(.,.):  ( a , b] \times G\rightarrow \mathbb{R}$ be a function such that $f(x,y(x)) \in L_q^1(a, b)$ for any $y \in G$,
$n-1<\alpha\leq n, n\in\mathbb{N}$, $0\leq \beta\leq 1$, $\gamma=(n-\alpha)(1-\beta)$ and assume that $I^\gamma_{q,a+}y \in AC_q^n[a, b]$.  Then $y(t)$ satisfies a.e. the relations (\ref{additive1.1})-(\ref{additive1.2}) if and only if $y(x)$ satisfies a.e. the integral equation
\begin{eqnarray}\label{additive4.11}
y(x):=\sum\limits_{k=0}^{n-1}\frac{b_k}{\Gamma_q\left(k-\gamma+1\right)}(x-a)_q^{k-\gamma}+
\left(I^\alpha_{q,a+}f(t,y(t))\right)(x).
\end{eqnarray}
\end{thm}
\begin{proof} \emph{Necessity.} Let $y(x) \in L^1_q[a, b]$ satisfy a.e. the relations (\ref{additive1.1})-(\ref{additive1.2}) and $f(x,y(x)) \in L_q^1(a, b)$ for any $y \in G$. Then  $\left(D^{\alpha,\beta}_{q,a+}y\right)(x)$ exists and belongs to $L^1_q[a, b]$ and, by Lemma \ref{lem 2.6}, we find that $\left(I^\alpha_{q,a+}f(t,y(t))\right)(x)\in L^1_q[a,b]$. By now applying the integral operator $I^\alpha_{q,a+}$ to both sides of (\ref{additive1.1}) and using the relation (\ref{additive4.5}) we obtain the equation (\ref{additive4.11}).  The necessity is proved.

\emph{Sufficiency.} Let $y(x) \in L^1_q[a, b]$ satisfy the equation (\ref{additive4.11}). Then, by applying the operator $D^{\alpha,\beta}_{q,a+}$  to both sides of (\ref{additive4.11}), we have that
\begin{eqnarray*}
\left(D^{\alpha,\beta}_{q,a+}y\right)(x)&=&\sum\limits_{k=1}^{n-1}\frac{b_k}{\Gamma_q\left(k-\gamma+1\right)}
\left(D^{\alpha,\beta}_{q,a+}(t-a)_q^{k-\gamma}\right)(x)\nonumber\\
&+&\left(D^{\alpha,\beta}_{q,a+}I^\alpha_{q,a+}f(t,y(t))\right)(x).
\end{eqnarray*}

Hence, by using Definition \ref{definition4.2}, (\ref{additive2.2}) and (\ref{additive2.8}), we find that
$\left(D^{\alpha,\beta}_{q,a+}(t-a)_q^{k-\gamma}\right)(x)=0$, for $k-\gamma=k-(n-\alpha)(1-\beta)=k-n+\alpha+\beta n-\alpha\beta < \alpha+\beta n-\alpha\beta, 0\leq k\leq n-1 $. Furthermore, it yields that
\begin{eqnarray*}
\left(D^{\alpha,\beta}_{q,a+}y\right)(x)&=& f(x,y(x)).
\end{eqnarray*}

Now we show that the relations (\ref{additive1.2}) also hold. For this we apply the operator $I^\gamma_{q,a+}$ to both sides of (\ref{additive4.11}) and use Lemma \ref{lem 2.3}   and (3.5) to conclude that
\begin{eqnarray}\label{additive4.12}
\left(I^\gamma_{q,a+}y\right)(x)&=&\sum\limits_{k=1}^{n-1}\frac{b_k}{\Gamma_q\left(k-\gamma+1\right)}
\left(I^\gamma_{q,a+}(t-a)_q^{k-\gamma}\right)(x)\nonumber\\
&+&\left(I^\gamma_{q,a+}I^\alpha_{q,a+}f(t,y(t))\right)(x)\nonumber\\
&=&\sum\limits_{k=1}^{n-1}\frac{b_k}{[k]_q!} (t-a)_q^k+\left(I^{n-n\beta+\alpha\beta}_{q,a+}f(t,y(t))\right)(x).
\end{eqnarray}

Let  $0\leq m\leq n-1$. Then, by using Definition \ref{definition4.2}, (\ref{additive2.2}), (\ref{additive2.8}) and  (\ref{additive1.2}), we obtain  that
\begin{eqnarray}\label{additive4.13}
D_q^m\left(I^\gamma_{q,a+}y\right)(x)&=&\sum\limits_{k=1}^{n-1}\frac{b_k}{[k-m]_q!} (t-a)_q^{k-m}\nonumber\\
&+&\frac{1}{\Gamma_q\left(n-n\beta+\alpha\beta-m\right)}\int\limits_a^x(x-qt)_q^{n-n\beta+\alpha\beta-m-1}f(t,y(t))d_qt.
\end{eqnarray}

Taking in (\ref{additive4.13}) the limit  $ x \rightarrow a+$ a.e., we obtain the relations in (\ref{additive1.2}). Thus also the sufficiency is proved, which completes the
proof.
\end{proof}

\subsection{  The existence and uniqueness of solutions to the Cauchy type $q$-fractional problem  (\ref{additive1.1})-(\ref{additive1.2}) in $L^1_{\alpha,\beta, q}[a,b]$}.

\vspace{0,3cm}
In this subsection we give conditions for a unique global solution to the Cauchy type
problem (\ref{additive1.1})-(\ref{additive1.2}) in the space $L^1_{\alpha,\beta, q}[a,b]$ defined for $\alpha>0$ by
\begin{eqnarray*}
L^1_{\alpha,\beta, q}[a,b]:=\left\{y\in L_q^1[a, b]: D^{\alpha,\beta}_{q,a+}y \in L_q^1[a, b]\right\}.
\end{eqnarray*}

\begin{thm}\label{thm 4.2}  Let $a > 0$,   $G\subset \mathbb{R}$ be an open set  and  $f(.,.) :
[a,b]\times G \rightarrow \mathbb{R}$ be a function such that $f(x,y(x))\in L_q^1[a,b]$ for any $y\in G$ and
satisfying the condition (\ref{additive3.6}). If $n-1<\alpha\leq n, n\in\mathbb{N}$, $0\leq \beta\leq 1$, $\gamma=(n-\alpha)(1-\beta)$, $I^\gamma_{q,a+}y \in AC_q^n[a, b]$, then there exists a unique solution $y(x)\in L^1_{\alpha,\beta, q}[a,b]$ to the Cauchy type problem (\ref{additive1.1})-(\ref{additive1.2}).

\begin{proof} We begin to prove the existence of a unique solution $y \in L_q^1[a, b]$. According to Theorem \ref{thm 4.1}, it is sufficient to prove the existence of a unique solution $y\in L_q^1[a, b]$ to the nonlinear Volterra integral equation (\ref{additive4.11}). Consequently, the equation (\ref{additive4.11}) can be written in the operator form $y = Fy$ such that
\begin{eqnarray}\label{additive4.14AA}
\left(Fy\right)(x):=y_0+\left[I^\alpha_{q,a+}f(t,y(t))\right](x),
\end{eqnarray}
where $y_0:=\sum\limits_{k=0}^{n-1}\frac{b_k}{\Gamma_q\left( k-\gamma+1\right)}(x-a)_q^{ k-\gamma}$.

Let $[a,\xi_1] \subset [a,b]$  be such that
\begin{eqnarray}\label{additive4.15AA}
\omega:=C\frac{\left(a-q\xi_1\right)_q^\alpha}{\Gamma_q\left(\alpha+1\right)}\leq1.
\end{eqnarray}

First we prove the following:  If $y\in L^1_q[a,\xi_1]$, then $(Ty)(x) \in L^1_q[a,\xi_1]$.  Indeed, since $y_0\in L^1_q[a,\xi_1]$, $f \left(x, y(x)\right) \in L^1_q[a,\xi_1]$, the integral on the right-hand side of (\ref{additive4.14AA}) also belongs to $L^1_q[a,\xi]$.
Hence, $(Ty)(x) \in L^1_q[a,\xi_1]$.

From (\ref{additive4.14AA}), (\ref{additive4.15AA}) and Lemma \ref{lem 2.6} it follows that
\begin{eqnarray*}
\| Fy_1-Fy_2\|_{L_q^1[a,\xi_1]}\leq C\|I^\alpha_{q,a+}\left(y_1-y_2\right)\|_{L_q^1[a,\xi_1]}\leq \omega\| y_1-y_2 \|_{L_q^1[a,\xi_1]},
\end{eqnarray*}
and the proof of our first claim is done. Since $L_q^1[a,\xi_1]$ is Banach space we are according to the Banach fixed point theorem (see e.g.\cite{KF1968}), there exists a unique solution $y'\in L_q^1[a,\xi_1]$ such that $Ty'=y'$.

Moreover, in view of this theorem, the solution $y'$ is obtained as a limit of a convergent sequence  $\left\{\left(F^my'_0\right)(x)\right\}$
\begin{eqnarray*}
\lim\limits_{m\rightarrow0}\|F^my'_0-y'\|_{L_q^1[a,\xi_1]}=0,
\end{eqnarray*}
in the space $L_q^1[a,\xi_1]$, where $ y_0$ is any function in $L_q^1[a,\xi_1]$.

If at least one $b_k\neq0$ in the initial condition
(\ref{additive1.2}), then we can take $y'_0=y_0$.  By (\ref{additive4.15AA}), the sequence $\left\{\left(F^my'_0\right)(x)\right\}$ is defined by
$$
\left(F^my'_0\right)(x):=y_0+\left[I^\alpha_{q,a+}f(t,F^{m-1}y'_0(t))\right](x),
$$
for $m\in\mathbb{N}$. Let $\left(y_m(x):=F^my'_0\right)(x)$. Then
$y_m(x)=y_0+\left[I^\alpha_{q,a+}f(t,y_{m-1} (t))\right](x)$ and
$$
\lim\limits_{m\rightarrow0}\|y_m-y'\|_{L_q^1[a,\xi_1]}=0.
$$

This means that we actually used the method of successive approximations to find a unique solution $y'(x)$ to the integral
equation (\ref{additive4.11}) on $[a,\xi_1]$. Next we consider the interval $[\xi_1, \xi_2]$, where $\xi_1= \xi_2 + h_1$, $h_1 > 0$ is such that $\xi_2 < \infty$. Rewrite the equation (\ref{additive4.11}) in the form
\begin{eqnarray}\label{additive4.16}
y(x)=y_0(x)+ \left(I^\alpha_{q,a+}f(t,y(t))\right)(\xi_1)+\left(I^\alpha_{q,\xi_1+}f(t,y(t))\right)(x).
\end{eqnarray}

Since the function $y(t)$ is uniquely defined on the interval $[a, \xi_1]$, the last integral can be considered as the known function,
and we can rewrite the last equation as
\begin{eqnarray}\label{additive4.14}
\left(Fy\right)(x):=y_{10}+\left[I^\alpha_{q,\xi_1+}f(t,y(t))\right](x),
\end{eqnarray}
where $y_{10}:=y_0(x)+ \left(I^\alpha_{q,a+}f(t,y(t))\right)(\xi_1)$.

In a similar way as above, we get  that there exists a unique solution $y' \in L^1_q[\xi_1, \xi_2]$
for (\ref{additive4.11}) on the interval $[\xi_1, \xi_2]$. Taking the next interval $[\xi_2, \xi_3]$, where $\xi_3 = \xi_2 + h_2$, $h_2 > 0$,$ \xi_3 <\infty$, and repeating the procedure, we conclude that there exists a unique solution $y' \in L^1_q[a, b]$  for (\ref{additive4.11})  and hence to the Cauchy type problem (\ref{additive1.1})-(\ref{additive1.2}).

Finally,    we must show that such a unique solution $y\in L^1_q[a, b]$ belongs to the space $L^1_{\alpha, \beta, q}[a, b]$. According to the above proof, the solution $y\in L^1_q[a, b]$ is a limit of the sequence $\left\{y_m\right\}\in L^1_q[a, b]$:
\begin{eqnarray}\label{additive4.15}
\lim\limits_{m\rightarrow\infty} \|y_m-y\|_{L^1_q[a, b]}=0,
\end{eqnarray}
with the choice of certain $y_m$ on each $[a, \xi_1], [\xi_1, \xi_2], \dots , [\xi_{L-1}, b]$. Since
\begin{eqnarray*}
\|D^{\alpha,\beta}_{q,a+}y_m-D^{\alpha,\beta}_{q,a+}y\|_{L^1_q[a, b]}=\|f\left(x,y_m\right)-f\left(x,y\right)\|_{L^1_q[a, b]}\leq A
\| y_m- y\|_{L^1_q[a, b]},
\end{eqnarray*}
by (\ref{additive4.15}), we obtain that
\begin{eqnarray*}
\lim\limits_{m\rightarrow\infty}\|D^{\alpha,\beta}_{q,a+}y_m-D^{\alpha,\beta}_{q,a+}y\|_{L^1_q[a, b]}=0,
\end{eqnarray*}
and hence $D^{\alpha,\beta}_{q,a+}\in L^1_q[a, b]$. This completes the proof.
\end{proof}

\end{thm}

\begin{cor}\label{cor 4.1}  Let $G$ be an open set in $\mathbb{R}$ and
$f(.,.):  ( a , b] \times G\rightarrow \mathbb{R}$ be a function such that $f(x,y(x)) \in L_q^1(a, b)$ for any $y \in G$. If
$0<\alpha\leq 1$, $0\leq \beta\leq 1$, $\gamma=(1-\alpha)(1-\beta)$, $I^\gamma_{q,a+}y \in AC_q^n[a, b]$,  then $y(t)$ satisfies a.e. the relations (\ref{additive1.3})-(\ref{additive1.14}) if and only if, $y(x)$ satisfies a.e. the integral equation
\begin{eqnarray}\label{additive4.14}
y(x):=b\frac{(x-a)_q^{k-\gamma}}{\Gamma_q\left(k-\gamma+1\right)}+
\left(I^\alpha_{q,a+}f(t,y(t))\right)(x).
\end{eqnarray}
\end{cor}

\begin{cor}\label{cor 4.2}Let $a > 0$,   $G\subset \mathbb{R}$ be an open set  and  $f(.,.) :
[a,b]\times G \rightarrow \mathbb{R}$ be a function such that $f(x,y(x))\in L_q^1[a,b]$ for any $y\in G$ and
satisfies the condition (\ref{additive3.6}). If $0<\alpha\leq 1$,  $0\leq \beta\leq 1$, $\gamma=(n-\alpha)(1-\beta)$, $I^\gamma_{q,a+}y \in AC_q^n[a, b]$. Then there exists a unique solution $y(x)\in L^1_{\alpha,\beta, q}[a,b]$ to the Cauchy type problem (\ref{additive1.3})-(\ref{additive1.14}).
\end{cor}

\section{  Examples}

In this section we present some  examples and discuss these examples  in connection with the results obtained in Sections 3 and 4. Our examples are $q$-analogues of examples given in \cite[Examples 3.1-3.2]{KST2006}.
\begin{exa} Let $\lambda,\alpha \in \mathbb{R}^+$, $0\leq\beta\leq 1$ and $\gamma \in\mathbb{R}$ be such that $-2\alpha-\gamma+1>0$. Then we consider
 $$
 f(x,y(x)):=\lambda   \frac{(x-q^{-\alpha-\gamma+1}a)_q^{\alpha+\gamma}}{(x-q^{-2\alpha-\gamma+1}a)_q^\alpha}  y^2(x),
 $$
and assume that $n-1<\alpha\leq n, n\in\mathbb{N}$, $0\leq \beta\leq 1$, $\gamma=(n-\alpha)(1-\beta)$ and $I^\gamma_{q,a+}y \in AC_q^n[a, b]$.

Moreover, we choose $G$ so that
$$
(x,y)\in[qa,b]\times G:=\left\{qa\leq x\leq b;  \left|y(x)\right|<M<\infty\right\}.
$$
We note that $f(x,y(x))$ satisfies (\ref{additive3.6}) since, for $-2\alpha-\gamma+1>0$,
$x\in[qa,b]$ and $y_1; y_2 \in G$,
\begin{eqnarray*}
\left|f(x,y_1(x))-f(x,y_2(x))\right|&=& \lambda \left|\frac{(x-q^{-\alpha-\gamma+1}a)_q^{\alpha+\gamma}}{(x-q^{-2\alpha-\gamma+1}a)_q^\alpha}  y^2_1(x)\right.\\
&-&\left.\frac{(x-q^{-\alpha-\gamma+1}a)_q^{\alpha+\gamma}}{(x-q^{-2\alpha-\gamma+1}a)_q^\alpha}  y^2_2(x) \right|\\
&\leq&\lambda\frac{(b-q^{-\alpha-\gamma+1}a)_q^{\alpha+\gamma}}{(a-q^{-2\alpha-\gamma+1}a)_q^\alpha} \left|y_1 (x)+y_2(x)\right|\left|y_1 (x)-  y_2(x)\right|\\
&\leq&\lambda\frac{(b-q^{-\alpha-\gamma+1}a)_q^{\alpha+\gamma}}{(a-q^{-2\alpha-\gamma+1}a)_q^\alpha}2M\left|y_1 (x)-  y_2(x)\right|,
\end{eqnarray*}
which proves that $f$ satisfies the Lipschitz condition (\ref{additive3.6}) in Theorem \ref{thm 4.2}.

Hence,
\begin{eqnarray*}
\int\limits_{qa}^b\left|f(x,y (x))\right|d_qx&=& \lambda \int\limits_{qa }^b\frac{(x-q^{-\alpha-\gamma+1}a)_q^{\alpha+\gamma}}{(x-q^{-2\alpha-\gamma+1}a)_q^\alpha}\left|y(x)\right|^2 d_qx\\
&\leq& \lambda \frac{(b-q^{-\alpha-\gamma+1}a)_q^{\alpha+\gamma}}{(a-q^{-2\alpha-\gamma+1}a)_q^\alpha} M^2,
\end{eqnarray*}
which means that $f(x,y (x))\in L_q^1[qa,b]$. We can conclude that, according to Theorem \ref{thm 4.2}, the equation

\begin{eqnarray}\label{additive5.1}
\left(D^{\alpha,\beta}_{q,qa+}y\right)(x)&=&  \lambda \frac{(x-q^{-\alpha-\gamma+1}a)_q^{\alpha+\gamma}}{(x-q^{-2\alpha-\gamma+1}a)_q^\alpha}  y^2(x) ,
\end{eqnarray}
\begin{eqnarray}\label{additive5.2}
\lim\limits_{x\rightarrow a+}\left(I^{(1-\alpha)(1-\beta)}_{q,a+}y \right)(x)=0,
\end{eqnarray}
has a unique solution. Note also that in this case, according to Lemma \ref{lem 2.3} the condition (\ref{additive1.2}) with $b_k=0$, $k=0,1,..n-1$ coincides with (\ref{additive5.2}).

Moreover, in this special case $2\alpha+\gamma<1$ we claim that this unique solution is given by
\begin{eqnarray}\label{additive5.2AAA}
y(x)=\frac{1}{\lambda} \frac{\Gamma_q\left(1-\alpha-\gamma\right)}{\Gamma_q\left(1-2\alpha-\gamma\right)}
\left(x- qa\right)^{-\alpha-\gamma}_q.
\end{eqnarray}

In fact, by applying the operator $D^{\alpha,\beta}_{q,qa+}$ and using in order (\ref{additive4.2}), (\ref{additive2.8}), (\ref{additive2.1}), (\ref{additive2.0}) and (\ref{additive2.8}), we get that
\begin{eqnarray*}
D^{\alpha,\beta}_{q,qa+}\left[\left(t- qa\right)^{-\alpha-\gamma}_q\right](x)  &=& \left(I_{q,a+}^{\beta(1-\alpha)}
 D_qI_{q,a+}^{ (1-\beta)(1-\alpha)}(t-qa)_q^{-\alpha-\gamma}\right)(x)\\
&=&\frac{\left(I_{q,a+}^{\beta(1-\alpha)}
\left(D_q\right)(t-qa)_q^{1-2\alpha-\beta+\alpha\beta-\gamma}\right)(x) }{\frac{\Gamma_q{(2-2\alpha-\beta+\alpha\beta-\gamma)}}{\Gamma_q{(1-\alpha-\gamma)}}} \\
&=&\frac{\left(I_{q,a+}^{\beta(1-\alpha)}
(t-qa)_q^{-2\alpha-\beta+\alpha\beta-\gamma}\right)(x)}
{\frac{\Gamma_q{(2-2\alpha-\beta+\alpha\beta-\gamma)}}{\left[1-2\alpha-\beta+\alpha\beta-\gamma\right]_q\Gamma_q{(1-\alpha-\gamma)}}}\\
&=&\frac{\left(I_{q,a+}^{\beta(1-\alpha)}
(t-qa)_q^{-2\alpha-\beta+\alpha\beta-\gamma}\right)(x)}{ \frac{\Gamma_q{(1-2\alpha-\beta+\alpha\beta-\gamma)}}{\Gamma_q{(1-\alpha-\gamma)}}}\\
&=&\frac{\frac{\Gamma_q{(1-2\alpha-\beta+\alpha\beta-\gamma)}}{\Gamma_q{(1+\beta-\alpha\beta-2\alpha-\beta+\alpha\beta-\gamma)}}
(x-qa)_q^{\beta-\alpha\beta-2\alpha-\beta+\alpha\beta-\gamma}}
{\frac{\Gamma_q{(1-\alpha-\gamma)}}{\Gamma_q{(1-2\alpha-\beta+\alpha\beta-\gamma)}}}\\
&=&\frac{\Gamma_q\left(1-\alpha-\gamma\right)}{\Gamma_q\left(1-2\alpha-\gamma\right)}
\left(x- qa\right)^{-2\alpha-\gamma}_q,
\end{eqnarray*}
for $2\alpha+\gamma<1$. Therefore,
\begin{eqnarray}\label{additive5.2A}
 D^{\alpha,\beta}_{q,qa+}\left[y(x)\right]
=\frac{1}{\lambda}\left[\frac{\Gamma_q\left(1-\alpha-\gamma\right)}{\Gamma_q\left(1-2\alpha-\gamma\right)}\right]^2
\left(x- qa\right)^{-2\alpha-\gamma}_q,
\end{eqnarray}

Hence, by using the well-known formulas
$$
(x-q^{-\alpha-\gamma+1}a)_q^{\alpha+\gamma}(x-qa)_q^{-\alpha-\gamma}=1 \;\; \text{and}\;\;
 (x-qa)_q^{-\alpha-\gamma}=(x-qa)_q^{-2\alpha-\gamma}(x-q^{-2\alpha-\gamma}a)_q^\alpha,
 $$
 in q-analysis (see [\cite{CK2000}, formula (3.7)]) we have that
\begin{eqnarray}\label{additive5.2B}
 f(x,y(x))&=&\lambda   \frac{(x-q^{-\alpha-\gamma+1}a)_q^{\alpha+\gamma}}{(x-q^{-2\alpha-\gamma+1}a)_q^\alpha}  y^2(x)\nonumber\\
&=&\frac{1}{\lambda}\left[\frac{\Gamma_q\left(1-\alpha-\gamma\right)}{\Gamma_q\left(1-2\alpha-\gamma\right)}\right]^2
\times\frac{(x-q^{-\alpha-\gamma+1}a)_q^{\alpha+\gamma}}{(x-q^{-2\alpha-\gamma+1}a)_q^\alpha}
\left[(x-qa)_q^{-\alpha-\gamma}\right]^2\nonumber\\
&=&\frac{1}{\lambda}\left[\frac{\Gamma_q\left(1-\alpha-\gamma\right)}{\Gamma_q\left(1-2\alpha-\gamma\right)}\right]^2
 (x-qa)_q^{-2\alpha-\gamma}=D_{q,qa+}^{\alpha,\beta}\left[y(x)\right],
\end{eqnarray}
i.e. $y(x)$ defined by (\ref{additive5.2AAA}) satisfies (\ref{additive5.1}).

Now, by combining (\ref{additive5.2A}) and using that  (\ref{additive5.2B}) and  $\lim\limits_{x\rightarrow qa}\left(x- qa\right)^{-\alpha-\gamma}_q=0$ for $x\in[qa,b]$,   we see that $y(x)$ defined by (\ref{additive5.2AAA}) satisfies (\ref{additive5.1})-(\ref{additive5.2}). The proof of our claim is complete.

\end{exa}

\begin{exa} Let $\alpha, \lambda\in\mathbb{R}^+$ and $0\leq\beta\leq 1$. Then we consider the following differential equation:
\begin{eqnarray}\label{additive5.3}
\left(D^\alpha_{q,qa+}y\right)(x) &=&  \lambda   \frac{\left[(x-qa )_q^{2\alpha+2\beta}\right]^\frac{1}{2}}
{(x-q^{\alpha+2\beta+1}a )_q^\alpha } \left[y(x)\right]^\frac{1}{2},
\end{eqnarray}
\begin{eqnarray}\label{additive5.4}
D^{\alpha-k}_{q,qa+}y(qa+)&=&0,
\end{eqnarray}
where $n-1<\alpha\leq{n}$, $n\in N$, $k=0,1,2,...,n-1$.

Let $f(x,y(x)):=\lambda   \frac{\left[(x-qa )_q^{2\alpha+2\beta}\right]^\frac{1}{2}}
{(x-q^{\alpha+2\beta+1}a )_q^\alpha } \left[y(x)\right]^\frac{1}{2}$ and assume $n-1<\alpha\leq n, n\in\mathbb{N}$, $0\leq \beta\leq 1$, $\gamma=(n-\alpha)(1-\beta)$ and $I^\gamma_{q,a+}y \in AC_q^n[a, b]$.

Moreover, let $G$ be such that
$$
(x,y)\in[qa,b]\times G:=\left\{a\leq x\leq b;  M_1< y(x) <M_2,   0<M_1,M_2\right\}.
$$

Then, for $x\in[a, b]$ and $y_1, y_2\in G$ we have that
\begin{eqnarray*}
|f(x,y_1(x))-f(x,y_2(x))|&\leq&\lambda\frac{\left[(b-qa )_q^{2\alpha+2\beta}\right]^\frac{1}{2}}
{(a-q^{\alpha+2\beta+1}a )_q^\alpha }\left|y^\frac{1}{2}_1(x)-y^\frac{1}{2}_2(x)\right|\\
&=&\lambda\frac{\left[(b-qa )_q^{2\alpha+2\beta}\right]^\frac{1}{2}}
{(a-q^{\alpha+2\beta+1}a )_q^\alpha }\frac{\left|y_1(x)-y_2(x)\right|}{y^\frac{1}{2}_1(x)+y^\frac{1}{2}_2(x)}\\
&\leq&\frac{\lambda}{2M_1}\frac{\left[(b-qa )_q^{2\alpha+2\beta}\right]^\frac{1}{2}}
{(a-q^{\alpha+2\beta+1}a )_q^\alpha }\left|y_1(x)-y_2(x)\right|,
\end{eqnarray*}
which proves that f satisfies the Lipschitz condition (\ref{additive3.6}) in Theorem \ref{thm 4.2}.

Moreover,
\begin{eqnarray*}
\int\limits_{qa}^b\left|f(x,y (x))\right|d_qx
 \leq  \lambda \frac{\left[(b-qa )_q^{2\alpha+2\beta}\right]^\frac{1}{2}}
{(a-q^{\alpha+2\beta+1}a )_q^\alpha }  M^\frac{1}{2},
\end{eqnarray*}
which means that $f(x,y (x))\in L_q^1[qa,b]$.  Therefore, by using Theorem \ref{thm 4.2}, we conclude that (\ref{additive5.3})-(\ref{additive5.4}) has a unique solution in $L_{q}^1[qa,b]$. By using similar argument as in our Example 5.1 we can in fact prove that this exact solution is:
\begin{eqnarray*}
y(x)=\left[\lambda  \frac{\Gamma_q\left(\alpha+2\beta+1\right)}{\Gamma_q\left(2\alpha+2\beta+1\right)}
\right]^2\left(x-qa\right)^{2\alpha+2\beta}_q.
\end{eqnarray*}
\end{exa}

\begin{rem}
  By arguing as in our Examples 5.1 and 5.2 we can also derive a similar $q$-analogue of Example 3.3 in \cite{KST2006}.
\end{rem}

\textbf{Acknowledgement:} We thank the referee for some general advices, which have improved the final version of the paper. The third author was supported by Ministry of Education and Science of the Republic of Kazakhstan Grant AP08052208.

\begin{center}

\end{center}

\end{document}